\documentclass[a4paper, 11pt]{article}

\usepackage{a4wide}
\usepackage[utf8]{inputenc}
\usepackage[english]{babel} 
\usepackage{amsmath}
\usepackage{amsfonts}
\usepackage{amssymb}
\usepackage{amsthm}
\usepackage{amscd}
\usepackage{euscript}
\usepackage{tikz}
\usepackage{url}
\usetikzlibrary{topaths,calc}
\usepackage{algorithm,algorithmic}
\usepackage{subfig}

\newtheorem{theo}{Theorem}[section]
\newtheorem{defi}[theo]{Definition}

\newtheorem{coro}[theo]{Corollary}
\newtheorem{lemma}[theo]{Lemma}

\newtheorem{conj}[theo]{Conjecture}

\DeclareMathOperator{\rank}{rank}
\DeclareMathOperator{\Ker}{Ker}
\DeclareMathOperator{\im}{Im}

\newcommand{\Z}{\mathbb{Z}}

\newcommand{\F}{\mathbb{F}}

\newcommand{\e}{\varepsilon}

\newcommand{\E}{{\EuScript E}}
\newcommand{\V}{{\EuScript V}}
\newcommand{\G}{{\EuScript G}}
\newcommand{\Esp}{{\mathbb E}}
\newcommand{\Prob}{{\mathbb P}}

\newcommand\mc[1]{\mathcal{#1}}

\newcommand\gen[1]{\langle {#1}\rangle}

\newenvironment{proofof}[2]{ \noindent{\it Proof of #1}~#2. }{\hfill\qed\medskip}

\title{A homological upper bound on critical probabilities for hyperbolic percolation}

\author{Nicolas Delfosse\thanks{Département de Physique, Université de Sherbrooke, Sherbrooke, Québec, J1K 2R1 Canada, nicolas.delfosse@usherbrooke.ca} \ and Gilles Z\'emor\thanks{Institut
    de Math\'ematiques de Bordeaux, UMR 5251, universit\'e de Bordeaux, zemor@math.u-bordeaux.fr}}

\begin{document}

\maketitle


\begin{abstract}
We study bond percolation for a family of infinite hyperbolic graphs. 
We relate percolation to the appearance of homology in finite versions
of these graphs. As a consequence, we derive an upper bound on the
critical probabilities of the infinite graphs.
\end{abstract}
%
%

\section{Introduction}

Let $\G=(\V,\E)$ be an infinite connected graph. Every edge is declared to be {\em
  open} with probability $p$, otherwise it is {\em closed}. This
endows subsets of edges with a product probability measure by
declaring edges to be open or closed independently of the others, and
creates a random open subgraph $\varepsilon$. If a given edge $e$
belongs to an infinite connected component of $\varepsilon$, we say
that (bond) percolation occurs. If the graph $\G$ is
edge-transitive, then the probability of percolation does not depend
on $e$ and we may denote this probability by $f(p)$. Arguably, the
most studied parameter of percolation theory is the {\em critical
  probability} $p_c=p_c(\G)$ which is the supremum of the set of $p$'s
for which $f(p)=0$.

Ever since the seminal work of Kesten \cite{K80}
percolation was extensively studied on the lattices associated to
$\Z^d$, for background see \cite{G89}: in the present paper, we are interested in percolation on
regular tilings of the hyperbolic plane. This topic was first
introduced by Benjamini and Schramm \cite{Be96}, and further studied
in \cite{Be01,GZ11,BMK09} among other papers. Specifically, our focus
is on the family of graphs that we shall denote by $G(m)$, for $m\geq
4$, that are regular of degree $m$, planar, and tile the plane by
elementary faces of length $m$. For $m=4$, the graph $G(m)$ is exactly
the square $\Z^2$ lattice. The local structure of the graph $G(5)$ is
represented in Figure \ref{fig:G(5)}.

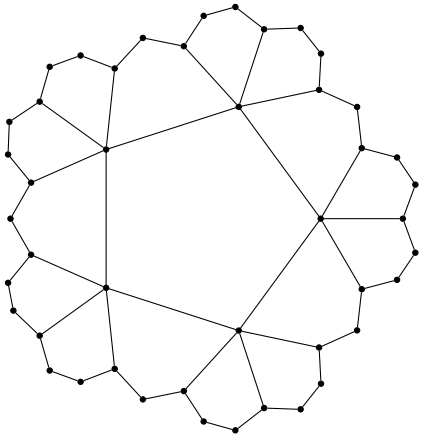
\begin{figure}[htbp]
  \centering
  \begin{tikzpicture}[scale=1.2]

\tikzstyle{every node}=[circle, draw, fill=black!100, inner sep=0pt, minimum width=2pt]

\draw
(0:1.3) \foreach \x in {72,144,...,359} { -- (\x:1.3) } -- cycle
(0:1.3) node (a) {}
(72:1.3) node (b) {}
(144:1.3) node (c) {}
(216:1.3) node (d) {}
(288:1.3) node (e) {};

\draw[shift=(a)]
(0:.9) node (a2) {}
(60:.9) node (a3) {}
(300:.9) node (a1) {};
\draw
(a1)--(a)--(a2)
(a)--(a3);

\draw[shift=(b)]
(72:.9) node (b2) {}
(132:.9) node (b3) {}
(12:.9) node (b1) {};
\draw
(b1)--(b)--(b2)
(b)--(b3);

\draw[shift=(c)]
(144:.9) node (c2) {}
(204:.9) node (c3) {}
(84:.9) node (c1) {};
\draw
(c1)--(c)--(c2)
(c)--(c3);

\draw[shift=(d)]
(216:.9) node (d2) {}
(276:.9) node (d3) {}
(156:.9) node (d1) {};
\draw
(d1)--(d)--(d2)
(d)--(d3);

\draw[shift=(e)]
(288:.9) node (e2) {}
(348:.9) node (e3) {}
(228:.9) node (e1) {};
\draw
(e1)--(e)--(e2)
(e)--(e3);

\draw
(36:2.1) node (f){}
(108:2.1) node (g){}
(180:2.1) node (h){}
(252:2.1) node (i){}
(324:2.1) node (j){};
\draw
(a3)--(f)--(b1)
(b3)--(g)--(c1)
(c3)--(h)--(d1)
(d3)--(i)--(e1)
(e3)--(j)--(a1);

\draw[shift=(a2)]
(70:.4) node (a24){}
(290:.4) node (a21){};
\draw[shift=(a3)]
(345:.4) node(a31){};
\draw[shift=(a1)]
(15:.4) node (a14){};
\draw
(a2)--(a24)--(a31)--(a3)
(a1)--(a14)--(a21)--(a2);

\draw[shift=(b2)]
(142:.4) node (b24){}
(2:.4) node (b21){};
\draw[shift=(b3)]
(57:.4) node(b31){};
\draw[shift=(b1)]
(87:.4) node (b14){};
\draw
(b2)--(b24)--(b31)--(b3)
(b1)--(b14)--(b21)--(b2);

\draw[shift=(c2)]
(214:.4) node (c24){}
(74:.4) node (c21){};
\draw[shift=(c3)]
(129:.4) node(c31){};
\draw[shift=(c1)]
(159:.4) node (c14){};
\draw
(c2)--(c24)--(c31)--(c3)
(c1)--(c14)--(c21)--(c2);

\draw[shift=(d2)]
(286:.4) node (d24){}
(136:.4) node (d21){};
\draw[shift=(d3)]
(201:.4) node(d31){};
\draw[shift=(d1)]
(231:.4) node (d14){};
\draw
(d2)--(d24)--(d31)--(d3)
(d1)--(d14)--(d21)--(d2);

\draw[shift=(e2)]
(358:.4) node (e24){}
(218:.4) node (e21){};
\draw[shift=(e3)]
(273:.4) node(e31){};
\draw[shift=(e1)]
(303:.4) node (e14){};
\draw
(e2)--(e24)--(e31)--(e3)
(e1)--(e14)--(e21)--(e2);

  \end{tikzpicture}
  \vspace*{13pt}
  \caption{The local structure of the graph $G(5)$}
  \label{fig:G(5)}
\end{figure}

Our goal is to study the critical probabilities of these
lattices. The simple lower bound $1/(m-1)\leq p_c$ can be derived
since $1/(m-1)$ is the critical probability for the $m$-regular tree,
and our main concern here is on dealing with upper bounds. Critical
probabilities for hyperbolic tilings
were studied numerically by Baek et al. \cite{BMK09} and also by Gu
and Ziff \cite{GZ11} who obtain a ``Monte Carlo'' upper bound
$p_c<0.34$ for $G(5)$. In 
previous work by the present authors \cite{DZ13}, the rigorous
upper bound $p_c < 0.38$ was obtained for $G(5)$ as a by-product of
the study of the erasure-correcting capabilities of a family of
quantum error-correcting codes. In the present paper we shall obtain
a substantially improved upper bound on critical probabilities that
gives $p_c<0.30$ for $G(5)$.

We remark that we restrict ourselves to the hyperbolic tilings $G(m)$
because they are self-dual and our method is better suited for this
case, but results on the critical probabilities for the self-dual case
can lead to results for the general case \cite{LB12}.

Classically, one uses finite portions of the infinite graph $\G$ to
devise intermediate tools for studying percolation. For example, in
the original $\Z^2$ setting, the standard (by now) method that leads
to the computation $p_c=1/2$ is to consider $n\times n$ finite grids
and study the probability of the appearance of an open path linking
the south boundary to the north boundary (or east to west) \cite{G89}.
In the hyperbolic setting however, trying to mimic this approach
directly quickly leads to serious obstacles: what finite portion of the
infinite graph $G(5)$ (say) should one consider, and which parts of the
boundary should be matched when looking for the appearance of finite
open paths ?
We shall overcome this difficulty by appealing to finite graphs
$G_t(m)$ that are everywhere locally isomorphic to $G(m)$, meaning
that every ball of radius $t$ of $G_t(m)$ is required to be isomorphic
to a ball of radius $t$ in the infinite graph $G(m)$. We shall derive
an upper bound $p_c\leq p_h$ on the critical probability by defining
a quantity $p_h$ such that, when $p>p_h$, then with probability
tending to $1$ when $t$ tends to infinity, $G_t(m)$ must contain 
an open cycle that can not be expressed as a sum modulo $2$ of
elementary faces. Our end result will be an expression for the upper
bound $p_h$ that involves only the structure of the infinite graph
$G(m)$, but the existence of the finite graphs $G_t(m)$ (which is
non-obvious) will be crucial to the derivation of $p_h$.

\paragraph{{\bf Outline and results:}}
Sections \ref{section:siran_graphs} and \ref{section:homology} are background.
In Section \ref{section:siran_graphs} we give a short description of a construction 
of the graphs $G_t(m)$ due to \v{S}ir\'a\v{n}. We shall need to
consider the cycles of those graphs that are not expressible as sums
of faces, i.e. that are homologically non-trivial: we shall therefore
need background on homology that is dealt with in Section \ref{section:homology}.

In Section \ref{section:rank} we study the appearance of homology in
random subgraphs of the finite graphs $G_t(m)$. We introduce a crucial
quantity $D(p)$ that we name the {\em rank difference function} and
that captures the limiting behaviour of the difference of the
dimensions of the homologies of the two random subgraphs of $G_t(m)$ chosen
through the parameters $p$ and $1-p$. We then define the quantity
$$p_h=\sup\left\{p, p-\frac 2m + D(p)=0\right\}.$$
The main result of this section, Theorem \ref{theo:Dequation}, is that
$p_h$ is an upper bound
on the critical probability of $G(m)$. We actually conjecture that for
$m\geq 5$ (i.e. the genuinely hyperbolic, or non-amenable, case)
this upper bound is also a lower bound, i.e. $p_c=p_h$. This would 
show that for these graphs the critical probability is local in a
sense close to \cite{Be11}. That $p_c\leq p_h$ was derived in
\cite{DZ13} in a roundabout way, through the study of the
erasure-decoding capabilities of quantum codes associated to the
tilings $G_t(m)$. The present proof not only removes the reference to
quantum coding, it is intrinsically shorter and more direct.

Section \ref{section:D(p)} is dedicated to finding an explicit
expression for the rank difference function $D(p)$, and hence for the
upper bound $p_h$. Our main result is Theorem~\ref{theo:D_graphical},
which expresses $D(p)$ as the series:
\begin{equation}
  \label{eq:D(p)}
  D(p) = \frac 2 m \sum_{ C } 
\left( \frac 1 {|V(C)|} \left(p^{|E(C)|} (1-p)^{|\partial(C)|} - (1-p)^{|E(C)|} p^{|\partial(C)|} \right) \right),
\end{equation}
where $C$ ranges over all connected subgraphs of $G(m)$ containing a
given vertex, where $V(C),E(C)$ denote the vertex and edge set of $C$,
and where $\partial (C)$ denotes the set of edges with at least one endpoint
in $C$, which are not in $E(C)$.
As mentioned, this expression for $D(p)$ does not involve the
graphs $G_t(m)$ anymore, but its proof crucially relies on their
existence.

Section \ref{section:approximation} proves that replacing
$D(p)$ in \eqref{eq:D(p)} by a truncated series continues to yield an upper
bound on the critical probability $p_c$ of $G(m)$
(Theorem~\ref{theo:bound}). This allows us to compute explicit
numerical upper bounds on $p_c$. Finally, Section \ref{section:conclusion}
summarizes the results with Theorem~\ref{theo:final} and gives some
concluding comments.









\section{Finite quotient of the regular hyperbolic tilings} \label{section:siran_graphs}
We are unaware of any method for constructing the required finite
versions of $G(m)$ that does not involve a fair amount of algebra.
In this section, we briefly recall \v{S}ir\'a\v{n}'s method to
construct such finite versions of the regular hyperbolic tiling $G(m)$.
The first step is to construct $G(m)$ from a group of matrices over a
ring of algebraic integers. 
Then this group is reduced modulo a prime number to yield the desired
finite graph.

Denote by $P_k(X) = 2\cos(k\arccos(X/2))$ the $k$-th normalized
Chebychev polynomial and let $\xi = 2\cos(\pi/m^2)$. Let $m \geq 5$ and consider the group $T(m)$ generated by the two following matrices of $SL_3(\Z[\xi])$.
$$
a =
\left(
\begin{array}{ccc}
P_m(\xi)^2-1 & 0 & P_m(\xi)\\
P_m(\xi) & 1 & 0\\
-P_m(\xi) & 0 & -1
\end{array}
\right)
\quad \text{and} \quad
b =
\left(
\begin{array}{ccc}
-1 & -P_m(\xi) & 0\\
P_m(\xi) & P_m(\xi)^2-1 & 0\\
P_m(\xi) & P_m(\xi)^2 & 1
\end{array}
\right).
$$
The group $T(m)$ admits the presentation
\begin{equation}
  \label{eq:presentation}
  T(m) = \gen{ a, b \ | \ a^m = b^m = (ab)^2 = 1 }.
\end{equation}

With this group we associate its \emph{coset graph}. The
coset graph associated with \eqref{eq:presentation}
is defined to be the infinite planar tiling whose vertex set, respectively edge set and face set, is the set of left cosets of the subgroup $\gen{a}$, respectively the set of left cosets of the subgroup $\gen{ab}$ and the subgroup $\gen{b}$. A vertex and an edge, or an edge and a face, are incident if and only if the corresponding cosets have a non-empty intersection.

For example, the coset $\gen{a} = \{1, a, a^2, \dots, a^{m-1} \}$ defines a vertex of the graph $G(m)$ and is incident to the $m$ edges represented by the cosets
$$
\gen{ab},\  a\gen{ab},\  a^2\gen{ab},\  \dots,\  a^{m-1} \gen{ab}.
$$
We can see that the coset graph is $m$-regular and that its faces
contain $m$ edges. It is straightforward to check that the coset graph associated with \eqref{eq:presentation} is the infinite planar graph $G(m)$ \cite{Si00}.

The basic idea to derive a finite version of this tiling is to reduce
the matrices defining the group $T(m)$ modulo a prime number. 
We can reduce the coefficients of the matrices of $T(m)$ thanks to the ring isomorphism $\Z[\xi] \simeq \Z[X]/h(X)$, where $h(X) \in \Z[X]$ is the minimal polynomial of the algebraic number $\xi$. This induces a ring morphism $\pi_p: SL_3(\Z[\xi]) \rightarrow SL_3(\F_p[X]/\bar h(X))$ where $\bar h(X)$ is the reduction modulo $p$ of the polynomial $h(X)$. Denote by $\bar T^p(m)$ the image of the group $T(m)$ by the morphism $\pi_p$.
The coset graph associated with the group $\bar T^p(m)$ is defined from the cosets of $\bar T^p(m)$, exactly like the coset graph of $T(m)$.

\v{S}ir\'a\v{n} proved that for a well chosen family of prime numbers
$p$, this construction provides a 
sequence of finite tilings $(G_t(m))_t$ which is locally isomorphic to
the infinite tiling $G(m)$ \cite{Si00}. Precisely:
\begin{theo}
\label{theo:siran}
For every integer $m \geq 5$, there exists a family of finite tilings
$(G_t(m))_{t \geq m}$ and some constant $K$ such that every ball of
radius $t$ of $G_t(m)$ is isomorphic to every ball of radius $t$ in
$G(m)$. Furthermore, the number of vertices of $G_t(m)$ is at most $K^t$.
\end{theo}

By construction, the graphs $G_t(m)$ are vertex transitive. Indeed, each element
of the group $\bar T(m)$ induces a graph automorphism of the coset
graph by left multiplication. An automorphism which sends a vertex $x
\gen{a}$ 
onto the vertex $y \gen{a}$ is given by the left multiplication by
$yx^{-1}$ of the cosets representing the vertices. For the same
reason, $G_t(m)$ is also edge-transitive and face-transitive.

To be sure that the faces of the graph $G_t(m)$ are not degenerate, we
require $t \geq m$. 
We will also use the fact that $G_t$ is a self-dual graph. This is a consequence of the local structure of the graph: every vertex has degree $m$ and every face has length $m$.


\section{Background on homology}\label{section:homology}


\subsection{Homology of a tiling of surface}
\label{subsection:homology}

A \emph{tiling of a surface} is a graph cellularly embedded in a smooth
surface. For us only the combinatorial structure of the surface plays
a role, therefore a face of the tiling is represented as the set of
edges on its boundary. 
We denote by $G=(V, E, F)$ such a tiling, where $F$ is the set of
faces that, as far as homology is concerned, can be thought of simply as a
privileged set of cycles of the graph $(V,E)$. 
With a tiling of a surface, we associate a \emph{dual tiling}
$G^*=(V^*,E^*,F^*)$. The vertices of this dual tiling are given by the faces of
$G$. Two vertices of $G^*$ are joined by an edge if the corresponding
faces of $G$ share an edge. Since every edge of $E$ belongs to exactly
two faces of $F$, there is a one-to-one
correspondence between edges of $G$ and edges of $G^*$. Finally, for
every vertex $v$ of $V$ the set of edges of $E$ incident to $v$
defines a face of $F^*$ through the above correspondence between $E$ and $E^*$.
 We assume the graph and its dual have neither multiple edges nor
 loops.
We shall also use $G$ to refer indifferently to
the graph $(V,E)$ and to the associated tiling $(V,E,F)$. 

In the remainder of this section, we consider only finite tilings, and we order the three sets $V, E$ and $F$ by $V = \{v_1,
v_2, \dots, v_{|V|}\}$, $E = \{e_1, e_2, \dots, e_{|E|}\}$ and $F =
\{f_1, f_2, \dots, f_{|F|}\}$. 
The \emph{incidence matrix} of the graph $(V,E)$ is defined to be the matrix $B(G) = (b_{ij})_{i, j}$ of $\mc M_{|V|, |E|}(\F_2)$ such that $b_{ij} = 1 $ if the vertex $v_i$ is incident to the edge $e_j$, and $b_{ij} = 0$ otherwise.

To emphasize the $\F_2$-linear structure of some subsets of $V$, $E$ and $F$, we introduce the spaces of \emph{$i$-chains} $C_i$:
$$
C_0 = \bigoplus_{v \in V} \F_2 v, \quad C_1 = \bigoplus_{e \in E} \F_2 e, \quad C_2 = \bigoplus_{f \in F} \F_2 f.
$$
In other words, the space $C_0 = \{\sum_v \lambda_v v \ | \ \lambda_v \in \F_2\}$ is the set of formal sums of vertices. The sets $C_1$ and $C_2$ are defined similarly.
These chain spaces are equipped with two $\F_2$-linear mappings
$
\partial_2: C_2 \rightarrow C_1
$
and
$
\partial_1: C_1 \rightarrow C_0
$
defined by $\partial_2(f) = \sum_{e \in f} e$ and $\partial_1(e) = \sum_{u \in e} u$.
These mappings are called \emph{boundary maps}.

A subset of the vertex set, respectively the edge set or the face set,
can be regarded as its indicator vector in $C_0$, respectively $C_1$
or $C_2$. 
This yields one-to-one correspondences between subsets and vectors,
which allow us to interpret geometrically the boundary maps. In subset
language, the map $\partial_2$ sends a subset of faces onto the set of
edges on its boundary in the standard sense, and the map $\partial_1$ sends a subset of edges
onto its ``endpoints'' which should be understood modulo $2$, i.e. the
set of vertices incident to an odd number of edges in the subset.

The singletons $\{v_i\}$, respectively $\{e_i\}$ and $\{f_i\}$, form a
basis of the space $C_0$, respectively $C_1$ and $C_2$. 
The matrix of the map $\partial_1$ in these singleton bases is equal
to the incidence matrix $B(G)$
of the graph $(V, E)$ and the matrix of the map $\partial_2$ is equal
to the transpose of the incidence matrix $B(G^*)$ of $(V^*,E^*)$.

We can easily prove that the composition of these applications is $\partial_1 \circ \partial_2 = 0$, implying the inclusion $\im \partial_2 \subset \Ker \partial_1$. We can now introduce the $\F_2$-homology of tilings of surfaces.
\begin{defi}
The \emph{first homology group} of a finite tiling of a surface $G$, denoted $H_1(G)$, is the quotient space
$$
H_1(G) = \Ker \partial_1 / \im \partial_2.
$$
\end{defi}

Note that $H_1(G)$ is also an $\F_2$-vector space.
The vectors of $\ker \partial_1$ are called \emph{cycles}. They
correspond to the subsets of edges that
meet every vertex an even number of times. The set $\ker \partial_1$
of cycles of a graph is an $\F_2$-linear space that we refer to as the
\emph{cycle code} of the graph. The vectors of $\im \partial_2$ are
called \emph{boundaries} or sums of faces and they describe the sets
of edges on the boundary of a subset of $F$.

In what follows, we shall study the dimension of the homology group of
different tilings of surfaces. The following well known property (see
e.g. \cite{Berge73} for a proof) is used repeatedly .
\begin{lemma} \label{lemma:Dcycle}
The dimension of the cycle code of a graph $G=(V, E)$ composed of $\kappa$ connected components, is
$
|E| - |V| + \kappa.
$
\end{lemma}



Figure~\ref{fig:torus}(a) represents a square lattice of the torus. A
cycle of trivial homology is drawn 
on Figure~\ref{fig:torus}(b). This cycle is clearly a sum of faces. 
Two examples of cycles with non trivial homology are given in Figure~\ref{fig:torus}(c) and (d). The first homology group of this tiling of the torus is a binary space of dimension 2. It is generated, for example, by an horizontal cycle which wraps around the torus, such as the one in Figure~\ref{fig:torus}(c) and a vertical cycle which wraps around the torus. The cycle of Figure~\ref{fig:torus}(d) is equivalent to the sum of these horizontal and vertical cycles, up to a sum of faces.

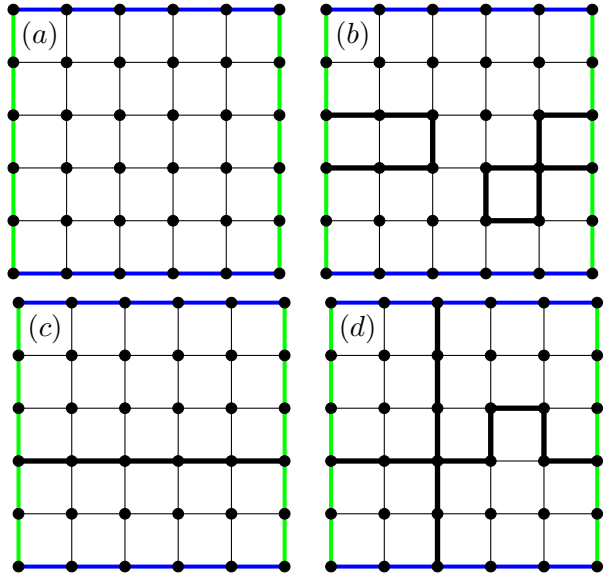
\begin{figure}[htbp]
\begin{center}
\begin{tikzpicture}[scale=.7]
\tikzstyle{every node}=[circle, draw, fill, inner sep=0pt, minimum width=4pt];

\draw[step=1cm] (0,0) grid (5,5);

\draw[color=blue, line width=1.5pt]
	(0,0)--(5,0)
	(0,5)--(5,5);

\draw[color=green, line width=1.5pt]
	(0,0)--(0,5)
	(5,0)--(5,5);

\foreach \x in {0,1,...,5}
	\foreach \y in {0,1,...,5}
		\draw (\x, \y) node {};
		
\tikzstyle{every node}=[inner sep=0pt, minimum width=4pt];
\node at (0.5,4.5) {$(a)$}; 
\end{tikzpicture}
\hspace{.2cm}
\begin{tikzpicture}[scale=.7]
\tikzstyle{every node}=[circle, draw, fill, inner sep=0pt, minimum width=4pt];

\draw[step=1cm] (0,0) grid (5,5);

\draw[color=blue, line width=1.5pt]
	(0,0)--(5,0)
	(0,5)--(5,5);

\draw[color=green, line width=1.5pt]
	(0,0)--(0,5)
	(5,0)--(5,5);

\foreach \x in {0,1,...,5}
	\foreach \y in {0,1,...,5}
		\draw (\x, \y) node {};
		
\draw[line width = 2pt]
	(5,2)--(3,2)--(3, 1)--(4,1)--(4,3)--(5,3)
	(0,2)--(2,2)--(2,3)--(0,3);
	
\tikzstyle{every node}=[inner sep=0pt, minimum width=4pt];
\node at (0.5,4.5) {$(b)$};
\end{tikzpicture}
\vspace{.2cm}

\begin{tikzpicture}[scale=.7]
\tikzstyle{every node}=[circle, draw, fill, inner sep=0pt, minimum width=4pt];

\draw[step=1cm] (0,0) grid (5,5);

\draw[color=blue, line width=1.5pt]
	(0,0)--(5,0)
	(0,5)--(5,5);

\draw[color=green, line width=1.5pt]
	(0,0)--(0,5)
	(5,0)--(5,5);

\foreach \x in {0,1,...,5}
	\foreach \y in {0,1,...,5}
		\draw (\x, \y) node {};
		
\draw[line width = 2pt]
	(0,2)--(5,2);
			
\tikzstyle{every node}=[inner sep=0pt, minimum width=4pt];
\node at (0.5,4.5) {$(c)$}; 
\end{tikzpicture}
\hspace{.2cm}
\begin{tikzpicture}[scale=.7]
\tikzstyle{every node}=[circle, draw, fill, inner sep=0pt, minimum width=4pt];

\draw[step=1cm] (0,0) grid (5,5);

\draw[color=blue, line width=1.5pt]
	(0,0)--(5,0)
	(0,5)--(5,5);

\draw[color=green, line width=1.5pt]
	(0,0)--(0,5)
	(5,0)--(5,5);

\foreach \x in {0,1,...,5}
	\foreach \y in {0,1,...,5}
		\draw (\x, \y) node {};
		
\draw[line width = 2pt]
	(0,2)--(3,2)--(3,3)--(4,3)--(4,2)--(5,2)
	(2,0)--(2,5);
	
\tikzstyle{every node}=[inner sep=0pt, minimum width=4pt];
\node at (0.5,4.5) {$(d)$};
\end{tikzpicture}

\caption{(a) A square tiling of the torus. The opposite boundaries are identified. (b) A cycle which is a boundary. (c) A cycle which is not a boundary. (d) A cycle which is not a boundary.}
\label{fig:torus}
\end{center}
\end{figure}

\subsection{Induced homology of a subtiling}
\label{section:subhomology}

Percolation theory deals with random subgraphs of a given graph. In what follows, we introduce the homology of a subgraph of a given tiling $G$.

The subgraphs that we consider are obtained by selecting a subset of
edges. Denote by $G = (V, E, F)$ a tiling of surface and let us
consider the subgraph $G_{\e}$ of $G$ whose vertex set is exactly $V$
and whose edge set is a given subset $\e$ of $E$. This graph is not
immediately endowed with a set of faces and with a homology group. 
The proper notion of homology for our purpose is obtained by
considering the boundaries of the tiling $G$ which are included in the
subgraph~$G_{\e}$. More precisely, the subset of edges $\e$ defines
the subspace $C_0^\e=C_0$, the subspace $C_1^\e$ of $C_1$ made up of
all formal sums of edges of $\e$, and the subspace $C_2^\e$ of $C_2$
made up of all those vectors of $C_2$
whose image under $\partial_2$ is included in $C_1^\e$. The mappings
$\partial_1^\e$ and $\partial_2^\e$ are defined as the restrictions of
$\partial_1$ and $\partial_2$ to $C_1^\e$ and $C_2^\e$.

\begin{defi}
Let $G=(V, E, F)$ be a tiling of a surface and let $\e \subset E$. 
The \emph{induced homology group} of $G_{\e}$ is the quotient space
$$
H_1(G_{\e}) = \Ker \partial_1^{\e} / ( \im \partial_2^\e).
$$
\end{defi}

For more detailed background on the homology of surfaces and their
tilings see \cite{Ha02, Gi10}.

\section{Appearance of homology in a random subgraph of $G_t$}\label{section:rank}

\subsection{Homology of a subgraph}

This section is devoted to the analysis of the induced homology of a
subgraph of $G_t(m)$. To lighten notation we omit the indices $m$ and
$t$ and write $G=G_t(m)$. 
Following the notation of Section~\ref{section:subhomology}, $\e$ denotes a subset of $E$ and $G_{\e}$ denotes the subgraph of $G$ induced by $\e$.

The decomposition of the graph $G_{\e}$ into connected components induces a partition of the edges of $\e$: 
the set $\e$ is the disjoint union of the subsets $\e_i \subset E$,
for $i=1, 2, \dots, r$ and where each set $\e_i$ is the edge set of a
connected component of $G_{\e}$. 
The following lemma proves that this decomposition of the graph
$G_{\e}$ induces a decomposition of its homology group.

\begin{lemma} \label{lemma:H_decomposition}
Let $\e = \cup_{i=1}^r \e_i$ be the partition of $\e$ derived from the
decomposition of the graph $G_{\e}$ 
into connected components. Then, the dimension of the first homology group of $G_{\e}$ is at most
$$
\dim H_1(G_{\e}) \leq \sum_{i=1}^r \dim H_1(G_{\e_i}).
$$
\end{lemma}

\begin{proof}
Remark that the chain space $C_1^\e$ decomposes as $C_1^\e = \oplus_i C_1^{\e_i}$. This leads to a similar decomposition of the cycle code of $G_\e$.
$$
\Ker \partial_1^{\e} = \bigoplus_{i=1}^r \Ker \partial_1^{\e_i}.
$$
However, the image of $\im \partial_2^\e$ has a slightly different structure.
First, the chain space $C_2^\e$ is has no similar decomposition but it
still contains the direct sum $\oplus_i C_2^{\e_i}$. 
Hence, the image of $\im \partial_2^\e$ contains the direct sum
$
\bigoplus_{i=1}^r \im \partial_2^{\e_i}
$
as a subspace. This implies
\begin{align*}
\dim H_1(G_{\e}) 
& = \dim \left( \bigoplus_{i=1}^r \Ker \partial_1^{\e_i} / \im \partial_2^{\e} \right) \\
& \leq \dim \left( \bigoplus_{i=1}^r \Ker \partial_1^{\e_i} / \bigoplus_{i=1}^r \im \partial_2^{\e_i} \right).
\end{align*}
To conclude, notice that this last quotient is exactly the direct sum $\oplus_i H_1(G_{t, \e_i})$.
\end{proof}

The next lemma proves that if $\e$ is composed of small clusters, then it covers no homology.
\begin{lemma} \label{lemma:smallclusters}
Let $G_{\e}$ be a connected subgraph of $G=G_t(m)$. If $\e$ contains at
most $t$ edges, then we have $H_1(G_{\e}) = \{0\}$.
\end{lemma}

\begin{proof}
Since $G_{\e}$ is connected and contains less than $t$ edges, it is
included in a ball of radius~$t$. 
From Theorem~\ref{theo:siran}, this ball is isomorphic with a ball of
the planar graph $G(m)$. But 
this ball is itself planar and in a planar graph, every cycle is a boundary. Thus the group $H_1(G_{\e})$ is trivial.
\end{proof}

The next lemma will allow us to compute the dimension of the induced
homology group of every subgraph $G_{\e}$ of $G=G_t(m)$.
Since a set $\e \subset E$ can be regarded as a subset of $E^*$, it
also defines a subgraph $G^*_{\e}$ of the graph $G^*$. Let us denote
by $\rank G_\e$ ($\rank G_\e^*$) the rank of an incidence matrix of
$G_\e$ (of  $G^*_{\e}$). By Lemma~\ref{lemma:Dcycle} these ranks do
not depend on the choice of the incidence matrix of the graph. The
dimension of the induced homology group is given by:

\begin{lemma} \label{lemma:dimH}
For every $\e \subset E$, we have
$$
\dim H_1(G_{\e}) = |\e| - \frac{2}{m} |E|  + 1 + \rank G^*_{\bar \e} - \rank G_{\e}.
$$
\end{lemma}

\begin{proof}
The group $H_1(G_{\e})$ is the quotient of the cycle code of $G_{\e}$
by $\im\partial_2^\e$, the set of boundaries of $G$ which are included in the subgraph $\e$.

By definition, the cycle code of $G_{\e}$ is the kernel of the map
$\partial_1^{\e}$. 
Moreover, the incidence matrix of $G_{\e}$ is a matrix of this linear
map. Therefore, the dimension of the cycle code of the subgraph
$G_{\e}$ is 
\begin{equation}
  \label{eq:ker}
  \dim\ker\partial_1^{\e}=|\e| - \rank G_{\e}.
\end{equation}
The set of boundaries of $G$ is the image of the map $\partial_2$. 
We noticed in Section~\ref{subsection:homology} that a matrix of the
map $\partial_2$ is given by the transpose of $B(G^*)$, the incidence
matrix of $G^*$. This means that the boundaries of $G$ correspond to
the sums of rows of $B(G^*)$. 
These are the vectors of the form $x B(G^*)$, where $x$ is a binary vector.

Consider the incidence matrix of $G^*_{\bar \e}$, where $\bar \e$
denotes the complement of $\e$ in $E$. 
This matrix can be obtained from $B(G^*)$ by selecting the columns indexed by the edges in $\bar \e$.
Let us define a map $\phi$ which sends a sum of rows of $B(G^*)$ onto
the same sum of rows in the matrix 
$B(G^*_{\bar \e})$. It is the map
\begin{align*}
\phi : \im \partial_2 &\longrightarrow C_1^{\bar \e}\\
x B(G^*) & \longmapsto x_{\bar \e} B(G^*_{\bar \e}),
\end{align*}
where $x$ is a row vector of $\F_2^{|V|}$ and $x_{\bar \e}$ is its
restriction to the columns indexed by the edges of $\bar \e$.
Then, the boundaries of $G$ included in $\e$, are exactly the vectors of the kernel of $\phi$.
The dimension of this space is
\begin{equation}
  \label{eq:kerphi}
  \dim\im\partial_2^\e=\dim \ker \phi = \dim\im \partial_2 - \dim\im\phi = \rank G^* - \rank G^*_{\bar \e}.
\end{equation}
Now $\rank G^*=\dim\im\partial_1^* =
|E^*|-\dim\ker\partial_1^*$. Applying Lemma~\ref{lemma:Dcycle} to the
dimension of the cycle code $\ker\partial_1^*$ of $G^*$ and the fact
that $G=G_t(m)$ is connected, we get $\rank G^* = |F| - 1 = (2/m)|E|
-1$. Injecting this last fact into \eqref{eq:kerphi}, we obtain,
together with \eqref{eq:ker}, the formula for $\dim
H_1(G_{\e})=\dim\ker\partial_1^\e - \dim\im\partial_2^\e$.
\end{proof}

\subsection{The rank difference function}

We now consider the probabilistic behaviour of
the induced homology of a random subgraph of $G_t=G_t(m)$. To get a
distribution which locally coincides with the distribution of
percolation events, the subset of edges $\e$ is chosen by selecting each edge of $G_t$ independently with probability $p$. This defines a random subgraph $G_{t, \e}$ of the graph $G_t$.

The intuition we follow is that if we are below the critical
probability of the graph $G(m)$, then most connected components
appearing in the random subgraph $G_{t, \e}$ should be small. Thanks
to Lemma~\ref{lemma:smallclusters}, 
these clusters do not support any non trivial homology. 
This implies that if $p < p_c(G(m))$ then the dimension of the induced
homology of $G_{t, \e}$ must be small. 
Conversely, if we compute, using Lemma~\ref{lemma:dimH}, the expected
dimension of $H_1(G_{t, \e})$ and find it to be large, we know that
$p$ must be above the critical probability $p_c$. 
These considerations lead us to introduce the following quantity.

\begin{defi}
The rank difference function associated with the family of graphs $(G_t)_t$ is defined to be
$$
D(p) = \limsup_t \Esp_p \left( \frac{\rank G^*_{t, \bar \e} - \rank G_{t, \e}}{|E_t|} \right).
$$
\end{defi}

The rank difference function satisfies the folowing equation when $p$
is below the critical probability of $G(m)$.
\begin{theo}
\label{theo:Dequation}
If $p<p_c(G(m))$ then the rank difference function associated with the family $(G_t)_t$ satifies
$$
p - \frac 2 m + D(p) = 0.
$$
\end{theo}

\begin{coro}\label{cor:phom}
  Defining $p_h = \sup\{p,p - \frac 2 m + D(p) = 0\}$  we have
  $p_c\leq p_h$.
\end{coro}

Assume that $p<p_c(G(m))$. 
By definition of the critical probability, 
for any fixed edge $e$ of the infinite graph $G(m)$, the probability
that $e$ is contained in an open connected component $C(e)$ of $G(m)$ of size
strictly larger than $t$ 
vanishes when $t \rightarrow \infty$. The following lemma shows that
we observe a similar behaviour in the finite graphs $G_t$. It will be
instrumental in proving Theorem~\ref{theo:Dequation}.
\begin{lemma}
\label{lemma:leaving_components}
For every $t \geq 0$, fix an edge $e_t$ of the graph $G_t$ and denote
by $C(e_t)$ its (possibly empty) connected component in the random
subgraph $G_{t, \e}$. 
Then, the probability that $C(e_t)$ contains strictly more than $t-2$
edges tends to $0$ when $t$ goes to infinity.
\end{lemma}

\begin{proof}
The complementary event depends only on what occurs inside the ball of
radius $t$ centered on an endpoint of the edge $e_t$. Since this ball
is isomorphic to the ball with the same radius in $G(m)$, this event
has the same probability in the space $G(m)$ and in $G_t(m)$. Hence
the result by the remark preceding the lemma.
\end{proof}

\begin{proofof}{Theorem}{\ref{theo:Dequation}}

Thanks to Lemma~\ref{lemma:H_decomposition}, we have the following upper bound on the dimension of the first homology group of $G_{t, \e}$:
$$
\dim H_1(G_{t, \e}) \leq  \sum_{i=1}^r \dim H_1(G_{t, \e_i}).
$$
where $\e_i$ is the edge set of the $i$-th connected component of $G_{t, \e}$.

From Lemma~\ref{lemma:smallclusters}, all the components $\e_i$ of size smaller than $t$ have a trivial contribution to $H_1(G_{t, \e})$. For the other components, the dimension of $H_1(G_{t, \e_i})$ is bounded by the number of edges in the component $\e_i$. Indeed, the induced homology group of $G_{t, \e_i}$ is a quotient of the cycle code of this graph, whose dimension is at most the number of egdes in $\e_i$.
This implies
$$
\dim H_1(G_{t, \e}) \leq |\{ e \in E_t \ \text{such that} \ |C(e)| > t  \}|,
$$
where $C(e)$ denotes the connected component in $G_{t, \e}$ of the
edge $e$ and $|C(e)|$ is its number of edges. 

Let us denote by $X_t=X_t(G_{t, \e})$ the cardinality of the set 
$\{ e \in E_t \ \text{such that} \ |C(e)| > t  \}$. 
To study the expectation of $X_t$, we define a random variable $X_e$, associated with each edge $e \in E_t$, which takes the value $X_e(G_{t, \e}) = 1$ if the size of $C(e)$ is larger than $t$ and which is 0 otherwise. Consequently, we have
$$
X_t = \sum_{e \in E_t} X_e,
$$
and by linearity of expectation, $\Esp(X_t) = \sum_e \Esp(X_e)$. For
every edge $e \in E_t$, this expectation of the random variable $X_e$
is $\Esp(X_e) = \Prob(|C(e)| > t)$. 
By edge-transitivity of the graph $G_t$, this quantity does not depend
on the edge $e$, thus $\Esp(X_t) = |E_t| \  \Prob(|C(e_t)| > t)$, for
some fixed edge $e_t$ of the graph $G_t$. Moreover, from
Lemma~\ref{lemma:leaving_components}, this probability vanishes 
when $t$ goes to infinity. 
This allows us to bound the expected dimension of the induced homology:
$$
\Esp_p \left( \frac{\dim H_1(G_{t, \e})}{|E_t|} \right) \leq \Esp_p \left( \frac{X_t}{|E_t|} \right) = \Prob_p( |C(e_t)| > t) \rightarrow 0.
$$
Since the right-hand side tends to 0 when $t$ goes to infinity, taking the superior limit gives exactly 0, \emph{i.e.} 
$$
\limsup_t \Esp_p \left( \frac{\dim H_1(G_{t, \e})}{|E_t|} \right) = 0.
$$

To conclude the proof, we determine the expected dimension of the induced homology group with the help of Lemma~\ref{lemma:dimH}. We find
$$
\limsup_t \Esp_p \left( \frac{\dim H_1(G_{t, \e})}{|E_t|} \right) = p - \frac 2 m + D(p).
$$
\end{proofof}

\section{Computation of the rank difference function of hyperbolic tilings}\label{section:D(p)}

The behaviour of the function $D(p)$ is difficult to capture directly
from its definition. The aim of this section is to provide an explicit
combinatorial description of the rank difference function $D(p)$
associated with the finite tilings $(G_t)_t$.

The next lemma enables us to replace the rank which appears in the
definition of $D(p)$ by a strictly graph-theoretical quantity.
\begin{lemma} \label{lemma:kappa}
Let $\kappa_{t, \e}$ denote the number of connected components of the graph $G_{t, \e}$. We have:
$$
\rank G_{t, \e} = |V_t| - \kappa_{t, \e}.
$$
\end{lemma}

\begin{proof}
By definition, the rank of the graph $G_{t, \e}$ is the rank of an incidence matrix of this graph. The kernel of this incidence matrix is the cycle code of the graph $G_{t, \e}$, which has dimension $|\e| - |V_t| + \kappa_{t, \e}$ from Lemma~\ref{lemma:Dcycle}. The result follows.
\end{proof}

The function $D(p)$ depends on the expected rank of the random submatrix  $G_{t, \e}$. This encourages us to examine the expected number of connected components of the random subgraph $G_{t, \e}$. A key ingredient of our study is the following decomposition of the random variable $\kappa_{t, \e}$.
\begin{lemma} \label{lemma:kappa_decomposition}
Let $C$ be a connected subgraph of $G_t$.  Denote by $X_C$ the random variable which takes the value 1 if $C$ is a connected component of the random graph $G_{t, \e}$ and 0 otherwise.
Then, we have
$$
\kappa_{t, \e} = \sum_{C \in \mc C_t} X_C
$$
where $\mc C_t$ denotes the set of connected subgraphs $C$ of $G_t(m)$.

Moreover, we have $\Esp_p(X_C) = p^{|E(C)|}(1-p)^{|\partial(C)|}$
where $\partial(C)$ is the set of edges of $G_t$ which are incident to at least one vertex of $C$, but which do not belong to $E(C)$.
\end{lemma}


The proof of the above lemma is self-evident.
Using this decomposition of $\kappa_{t, \e}$, we derive the following exact expression of the rank difference function as a function of the subgraphs of the infinite graph $G(m)$.

\begin{theo} \label{theo:D_graphical}
For $m\geq 5$ and $0<p\leq 1/2$,
The rank difference function associated with the graphs $(G_t(m))_t$ is equal to
$$
D(p) = \frac 2 m \sum_{ C \in \mc C(v) } 
\left( \frac 1 {|V(C)|} \left(p^{|E(C)|} (1-p)^{|\partial(C)|} - (1-p)^{|E(C)|} p^{|\partial(C)|} \right) \right),
$$
where $\mc C(v)$ denotes the set of connected subgraphs $C$ of $G(m)$
containing a fixed vertex $v$.
\end{theo}

\begin{proof}
From Lemma~\ref{lemma:kappa}, the rank difference function can be rewritten
\begin{align*}
D(p) & = \limsup_t \Esp_p\left( \frac{\kappa_{t, \e} - \kappa_{t, \bar \e}} {|E_t|} \right)\\
&= \limsup_t \left( \Esp_p \left( \frac{\kappa_{t, \e} } {|E_t|} \right) - \Esp_{1-p} \left( \frac{\kappa_{t, \e} } {|E_t|} \right) \right).
\end{align*}
where we used the fact that, $\bar \e$ being the complement of $\e$ in
$E_t$, we have $\Esp_p(\kappa_{t, \bar \e}) = \Esp_{1-p}(\kappa_{t, \e})$.

Then, using the decomposition of $\kappa_{t, \e}$ proposed in Lemma~\ref{lemma:kappa_decomposition} and the linearity of expectation, we obtain
$$
D(p) = \limsup_t \frac 1 {|E_t|} \sum_{ C \in \mc C_t } \left( \Esp_p(X_C) - \Esp_{1-p}(X_C)\right).
$$\medskip

{\em Elimination of the large components---}
Now, remark that the main contribution in this sum is given by the small components. To prove this, consider a sequence of integers $(M_t)_t$ such that $M_t \rightarrow +\infty$. Then, we have
\begin{align*}
\frac 1 {|E_t|} \sum_{\substack{C \in \mc C_t \\ |E(C)| \geq M_t}} \left( \Esp_p(X_C) - \Esp_{1-p}(X_C)\right)
& \leq \frac 1 {|E_t|} \sum_{\substack{C \in \mc C_t \\ |E(C)| \geq M_t}} \left( \Esp_p(X_C) + \Esp_{1-p}(X_C)\right)\\
& = \frac 1 {|E_t|} 
\Esp_p \left( \sum_{\substack{C \in \mc C_t \\ |E(C)| \geq M_t}} X_C \right) +
\frac 1 {|E_t|} 
 \Esp_{1-p} \left( \sum_{\substack{C \in \mc C_t \\ |E(C)| \geq M_t}} X_C\right)\\
& \leq \frac 1 {|E_t|} \frac {2 |E_t|}{M_t} = \frac {2}{M_t} \rightarrow 0
\end{align*}
To obtain the last inequality, remark that the sum of all the random
variables $X_C$ such that $|E(C)| \geq M_t$ counts the number of
connected components of the subgraph $G_{t, \e}$ of size larger than
$M_t$. Since connected components are disjoint, this number cannot be larger than $|E_t|/M_t$.

The previous paragraph proves that, for every sequence $M_t$ going to infinity, the rank difference function is given by
$$
D(p) = \limsup_t \frac 1 {|E_t|} \sum_{\substack{C \in \mc C_t \\ |E(C)| < M_t}} \left( \Esp_p(X_C) - \Esp_{1-p}(X_C)\right)
$$\medskip

{\em Recentralization---}
In order to remove the dependency on $t$, we would like to apply the local isomorphism between $G_t(m)$ and $G(m)$ and to express everything as a function of the infinite graph $G(m)$. 
First, we have to recenter all the components $C$ around a fixed vertex $v_t$ of the graph $G_t$.
To move a connected component $C$ of the graph $G_t$ onto a component which contains the vertex $v=v_t$, we use a family of automorphisms of the graph $G_t(m)$. For every vertex $w$ of the graph $G_t(m)$, select $\sigma_{v, w}$, an automorphism of the graph $G_t(m)$ sending $v$ onto $w$. We take the identity for $\sigma_{v, v}$. Such an automorphism exists because the graph $G_t$ is vertex transitive, as explained in Section~\ref{section:siran_graphs}.
From this fixed family of automorphisms, we can reach all the connected subgraphs of $G_t$, starting from the subgraphs containing $v$. Stated differently, we have
$$
\mc C_t =\{ C \ | \ C \text{ connected } \} = \bigcup_{w \in V_t} \{ \sigma_{v, w}(C) \ | \   C \text{ connected }, v \in V(C) \}
$$
At the right-hand side of this equality, each component $C$ of the graph appears $|V(C)|$ times.
Moreover, the contribution $\Esp_p(X_C)$ of the subgraph $C$, 
computed in Lemma~\ref{lemma:kappa_decomposition}, 
depends only on $|E(C)|$ and $|\partial(C)|$, which are both invariant under the application of an automorphism~$\sigma_{v, w}$. Hence, $D(p)$ is equal to
\begin{align*}
D(p) 
& = \limsup_t \frac 1 {|E_t|} \sum_{\substack{C \in \mc C_t \\ |E(C)| < M_t}} \left( \Esp_p(X_C) - \Esp_{1-p}(X_C)\right)\\
& = \limsup_t \frac 1 {|E_t|} \sum_{\substack{C \in \mc C_t(v) \\ |E(C)| < M_t}} \sum_{w \in V_t} \frac{1}{|V(C)|} \left( \Esp_p(X_{\sigma_{v,w}(C)}) - \Esp_{1-p}(X_{\sigma_{v,w}(C)})\right)\\
& = \limsup_t \frac 1 {|E_t|} \sum_{\substack{C \in \mc C_t(v) \\
    |E(C)| < M_t}} \frac {|V_t|} {|V(C)|}  \left( \Esp_p(X_C) -
  \Esp_{1-p}(X_C)\right)\\
& = \limsup_t \frac 2m \sum_{\substack{C \in \mc C_t(v) \\
    |E(C)| < M_t}} \frac {1} {|V(C)|}  \left( \Esp_p(X_C) -
  \Esp_{1-p}(X_C)\right)
\end{align*}
where we have used $\frac{|V_t|}{|E_t|}=\frac 2m$ since $G_t$ is $m$-regular. \medskip

{\em Application of the local isomorphism---}
We now replace the graph $G_t(m)$ by the infinite graph $G(m)$.
Since the balls of radius $t$ are isomorphic in $G_t(m)$ and in
$G(m)$, we have that every fixed subgraph $C$ inside such a ball has the same probability
of being a connected component whether it is of the random subgraph $G_{t,\e}$ or of
the open subgraph of $G(m)$.
By choosing $M_t=t-1$, we therefore get
\begin{equation}
  \label{eq:limsup}
 D(p) = \limsup_t \frac 2 m \sum_{\substack{C \in \mc C(v) \\ |E(C)| < M_t}} \frac {1} {|V(C)|}  \left( \Esp_p(X_C) - \Esp_{1-p}(X_C)\right)
\end{equation}
where $\mc C(v)$ denotes the set of connected subgraphs $C$ of $G(m)$ 
containing the fixed vertex $v$. \medskip

We can now conclude the proof. From
Lemma~\ref{lemma:kappa_decomposition}, the quantity $\left(
  \Esp_p(X_C) - \Esp_{1-p}(X_C)\right)$ is equal to $\left(p^{|E(C)|}
  (1-p)^{|\partial(C)|} - (1-p)^{|E(C)|} p^{|\partial(C)|} \right)$, which
is positive by Lemma~\ref{lemma:series_positivity} to be proven just
below. 
Therefore all the terms of the sum in \eqref{eq:limsup} are positive, which means that the $\limsup$ is in fact a limit. Since $M_t \rightarrow +\infty$, we get
$$
D(p) = \frac 2 m \sum_{ C \in \mc C(v) } 
\left( \frac 1 {|V(C)|} \left(p^{|E(C)|} (1-p)^{|\partial(C)|} - (1-p)^{|E(C)|} p^{|\partial(C)|} \right) \right).
$$
\end{proof}

It remains to prove that the series has positive terms. 
This result relies on an isoperimetric inequality.
\begin{lemma} \label{lemma:series_positivity}
Let $0<p<1/2$. For every connected subgraph $C$ of $G(m)$, we have 
$$
 p^{|E(C)|}(1-p)^{|\partial(C)|} - (1-p)^{|E(C)|} p^{|\partial(C)|} > 0.
$$
\end{lemma}

\begin{proof}
The parameter $p$ is assumed to be smaller than $1/2$. Thus, to prove that this quantity is strictly positive it suffices to show that for every connected subgraph $C$ of $G(m)$, we have $|E(C)| < |\partial(C)|$.
This inequality is somewhat analogous to the isoperimetric inequality that we recall now. The isoperimetric constant of the graph $G(m)$ is defined to be
$$
i_E(G(m)) = \inf \left\{ \frac{|\partial(C)|}{|V(C)|} \right\}
$$
with $C$ ranging over all finite subgraphs (that can be assumed
connected) of $G(m)$.
This number was computed exactly for hyperbolic graphs in \cite{HJL02}. It is
\begin{equation} \label{eqn:i_E}
i_E(G(m)) = (m-2) \sqrt{1-\frac{4}{(m-2)^2}}.
\end{equation}
In order to apply this to our problem, we write
\begin{align} \label{eqn:isoperimetric}
\frac{|\partial(C)|}{|E(C)|} = \frac{|\partial(C)|}{(m/2)|V(C)|-(1/2)|\partial(C)|} \geq \frac {i_E(G(m))}{m/2 - i_E(G(m))/2}
\end{align}
where we have used the fact that the smallest rate $|\partial(C)|/|E(C)|$ is achieved when $\partial(C)$ contains only edges with exactly one endpoint in $C$. In that case, we have
$m|V(C)| = 2|E(C)| + |\partial(C)|$.
Using Equation~(\ref{eqn:i_E}) and (\ref{eqn:isoperimetric}), it is then easy to check that, for all $m \geq 5$, we have
$$
\frac{|\partial(C)|}{|E(C)|} \geq \frac {i_E(G(5))}{5/2 - i_E(G(5))/2} \approx 1.62 > 1.
$$
This proves the lemma.
\end{proof}

\section{Bound on the critical probability of the hyperbolic lattice $G(m)$}\label{section:approximation}

We showed in Theorem~\ref{theo:Dequation} that the critical
probability of $G(m)$ is bounded from above as $p_c(G(m))\leq p_h$
with $p_h$ defined in
Corollary~\ref{cor:phom}. Theorem~\ref{theo:D_graphical} provides an
exact formula for the rank difference function $D(p)$ as a sum of a
series depending on the connected subgraphs of $G(m)$. This gives a
new expression for $p_h$ that does not involve the finite graphs
$G_t(m)$ anymore, but it still leaves $p_h$ difficult to compute. We
now show that by replacing the series $D(p)$ by its partial sums,
we obtain explicit upper bounds on $p_h$ and hence on $p_c$.

\begin{theo} \label{theo:bound}
Let $n \geq 0$ and let $D_n(p)$ be a partial sum of the series $D(p)$ associated with the hyperbolic graph $G(m)$. Then, the solution $p_h(n) \in [0, 1]$ of the equation
$$
p - 2/m + D_n(p)=0
$$
is an upper bound on $p_h$ and hence on $p_c(G(m))$.
\end{theo}

\begin{proof}
We have seen in Lemma~\ref{lemma:series_positivity} that all the terms
of the series $D(p)$ are strictly positive when $p>0$. Thus, every
partial sum $D_n(p)$ 
satisfies $D_n(p) < D(p)$. As a consequence, if $p_h(n)$ is a solution of the equation $p - 2/m + D_n(p)=0$, then we have $p_h(n) - 2/m + D(p_h(n))>0$. This proves that $D(p)$ does not satisfy the criterion of Theorem~\ref{theo:Dequation} at $p=p_h(n)$. Therefore $p_h(n)$ is an upper bound on $p_h$.
\end{proof}

As a first application of this theorem, using only the fact that $D_n(p) \geq 0$, we recover the upper bound $p_c(G(m)) \leq 2/m$, proved in \cite{DZ10}.

The first terms of the series, corresponding to the components of small size can be computed easily. For example the number of connected subgraphs of size 0, that is with 0 edges, containing a fixed vertex of $G(m)$ is 1 and this subgraph has a boundary $\partial(C)$ of size $m$.
This gives the partial sum
$$
D_0(p) = \frac 2 m ( (1-p)^m-p^m ).
$$
Applying Theorem~\ref{theo:bound} to $D_0(p)$, we get an upper close to $0.35$.
This is already more precise than the upper bound in \cite{DZ13}.

The next partial sum is given by
$$
D_1(p) = D_0(p) + \frac 2 m \left( \frac m 2 ( p(1-p)^{2(m-1)} - p^{2(m-1)}(1-p) ) \right),
$$
since there are $m$ different connected subgraphs of $G(m)$ composed
of one edge and containing a fixed vertex.

The first terms can be computed easily in this way. In a tree it is
possible to get an exact formula for the number of rooted connected
subgraphs using the Lagrange inversion threorem. However this
enumeration problem becomes extremely difficult when the subgraphs
start covering cycles. Moreover, the size of the boundary and the
number of vertices of the subgraph do not depend only on its number of
edges. We enumerated all the connected subgraphs of $G(5)$ (hyperbolic
animals, as in \cite{MW10}) of size at
most 8 by computer. The results are given in Table~\ref{tab:components}.
Using the partial sum $D_8(p)$ that takes into acount all the
subgraphs of size at most 8, we get an upper bound on $p_c(G(5))$ which is approximately $0.299973$:
$$
p_c(G(5)) \leq 0.299973.
$$
To the best of our knowledge, the previous best upper bound was close to $0.38$ \cite{DZ13}.
Gu and Ziff proposed a Monte-Carlo estimation of this threshold of $0.265$ \cite{GZ11} which is coherent 
with our upper bound.

\begin{table}[htbp]
\centering
\footnotesize
\caption{Enumeration of the rooted subgraphs of $G(5)$ up to size 8.}
\label{tab:components}
\begin{tabular}{c c c c}
\hline
$|E(C)|$ & $|V(C)|$ & $\partial(C)$ & occurrence\\
\hline
0 & 1 & 5 & 1\\
1 & 2 & 8 & 5\\
2 & 3 & 11 & 30\\
3 & 4 & 14 & 200\\
4 & 5 & 17 & 1400\\
4 & 5 & 16 & 25\\
5 & 6 & 20 & 10146\\
5 & 6 & 19 & 450\\
5 & 5 & 15 & 5\\
6 & 7 & 23 & 75460\\
6 & 7 & 22 & 5775\\
6 & 6 & 18 & 90\\
7 & 8 & 26 & 572720\\
7 & 8 & 25 & 64200\\
7 & 8 & 24 & 480\\
7 & 7 & 21 & 1155\\
8 & 9 & 29 & 4418190\\
8 & 9 & 28 & 661950\\
8 & 9 & 27 & 13005\\
8 & 8 & 24 & 12840\\
8 & 8 & 23 & 180\\
\hline
\end{tabular}
\end{table}

\section{Concluding comments}\label{section:conclusion}

Summarising Theorems \ref{theo:Dequation} and \ref{theo:D_graphical}
we have proved~:

\begin{theo}\label{theo:final}
  For $m\geq 5$ we have $p_c(G(m))\leq p_h$ with 
  \begin{align*}
    p_h  &= \sup \{ p \in [0, 1/2] \ | \ D(p)+p-\frac 2m = 0 \} \ \text{and}\\
    D(p) &= \frac 2 m \sum_{ C \in \mc C(v) } 
\left( \frac 1 {|V(C)|} \left(p^{|E(C)|} (1-p)^{\partial(C)} - (1-p)^{|E(C)|} p^{\partial(C)} \right) \right)
  \end{align*}
where $\mc C(v)$ denotes the set of connected subgraphs $C$ of $G(m)$
containing a fixed vertex $v$ of the graph $G(m)$.
\end{theo}

The value $p_h$ can be thought of as a critical value for the
appearance of homology in the graph $G(m)$. It captures the following
threshold~: for $p>p_h$, open subgraphs of large finite versions of
$G(m)$ must have a first homology group of dimension that scales linearly with the
total number of edges of the finite graph. For $p<p_h$, the dimension
of the homology group is sublinear instead. This bound is really
meaningful only for the hyperbolic case $m\geq 5$ since for $m=4$ (the square lattice), the
dimension of the total homology group of finite versions of the
infinite grid (tori) is limited to $2$.

A consequence of Theorem~\ref{theo:final} is that $p_h$ gives an upper
bound on the parameters of the quantum erasure channel that hyperbolic
surface codes built on the family $G_t(m)$ can sustain~\cite{DZ13}.

We conjecture~:

\begin{conj} \label{conj:p_c=p_h}
For $m\geq 5$, $p_c=p_h$.
\end{conj}

Recall that in hyperbolic lattices it has been shown that immediately
beyond the critical probability, the open subgraph contains infinitely
many infinite connected components \cite{Be96}. The conjecture could be seen as a
``finite'' (but unbounded) version of this fact.

\section*{Acknowledgements}

Nicolas Delfosse was supported by the Lockheed Martin Corporation. The
authors wish to thank Robert Ziff for his comments and Russell Lyons
for pointing out an inaccuracy in 
Lemma~\ref{lemma:H_decomposition} in a preliminary version of this article.


\newcommand{\SortNoop}[1]{}

\end{document}